\documentclass[11pt,reqno]{amsart}
\usepackage[brazilian,english]{babel} %Hifenização
\usepackage[T1]{fontenc}
\usepackage{amsfonts,amsthm,amssymb,amsmath}  %Fontes matemática e símbolos
\usepackage{graphicx,color}       %Inserir gráficos
\usepackage{latexsym}       %Símbolos especiais
\usepackage{hyperref}
\usepackage{graphicx}
\usepackage{multicol}
\usepackage{color}
\newtheorem{teore}{Theorem}[section]
\newtheorem{obs}[teore]{Remark}%[section]
\newtheorem{defi}{Definition}[section]

\newtheorem{coro}[teore]{Corollary}%[section]

\newtheorem{lem}[teore]{Lemma}%[section]

\newcommand{\ka}{\kappa}
\newcommand{\R}{\mathbb{R}}

\newcommand{\C}{\mathbb{C}}

\numberwithin{equation}{section}

\title[On a quadratic NLS system in dimension $n=5$]{On the dynamics of a quadratic Schr\"odinger system in dimension $n=5$}

\author[N. Noguera]{Norman Noguera}
\address{IMECC-UNICAMP, Rua S\'ergio Buarque de Holanda, 651, 13083-859, Cam\-pi\-nas-SP, Bra\-zil}
\email{nnoguera57@gmail.com}

\author[A. Pastor]{Ademir Pastor}
\address{IMECC-UNICAMP, Rua S\'ergio Buarque de Holanda, 651, 13083-859, Cam\-pi\-nas-SP, Bra\-zil}
\email{apastor@ime.unicamp.br}
\date{}
\subjclass[2010]{35Q55,  35B44, 35J47, 35A01}
\keywords{Global well-posedness; Blow up; Ground states solutions}

\begin{document}

\maketitle	

%\linenumbers
\begin{abstract}
In this work we give a sharp criterion for the global well-posedness, in the energy space, for a system of nonlinear Schr\"{o}dinger equations with quadratic interaction in dimension $n = 5$.  The criterion is given in terms of the charge and energy of the ground states associated with the system, which  are obtained by minimizing a Weinstein-type functional. The main result is then obtained in view of a sharp Gagliardo-Nirenberg-type inequality.
\end{abstract}

\section{Introduction}
In this notes we are interested in the following initial-value problem
\begin{equation}\label{system1}
\begin{cases}
\displaystyle i\partial_{t}u+\frac{1}{2m}\Delta u=\lambda\overline{u}v\\\\
\displaystyle i\partial_{t}v+\frac{1}{2M}\Delta v=\mu u^{2},\\\\
(u(x,0),v(x,0))=(u_0,v_0).
\end{cases}
\end{equation}
where $u,v:\R^n\times \R\to \C$, $(x,t)\in \R^n\times \R$, $\Delta$ is the Laplacian operator, $m,M>0$ are real constants and $\lambda,\mu \in \C$. System \eqref{system1} can be regarded as a non-relativistic limit of a Klein-Gordon system under the so called mass resonance condition $M=2m$. Also,  similar systems appear as models for the interaction of waves propagating in a $\chi^{(2)}$ dispersive media (see, for instance, \cite{colin}). From the mathematical point of view, the study of nonlinear  Schr\"{o}dinger systems with quadratic interaction has been increasing in recent years.  To cite a few, we refer the reader to \cite{colin}, \cite{corcho}, \cite{Hayashi3}, \cite{Hayashi4},  \cite{Hayashi}, \cite{hoshi}, \cite{Li}, \cite{LiHa}, \cite{oza}, and references therein.

An almost complete study of system \eqref{system1} in $L^2(\R^n)$ and $H^1(\R^n)$ was undertaken in \cite{Hayashi} (see also \cite{Hayashi3}, \cite{Hayashi4}). First of all, one should note that $L^2(\R^n)$ and $H^1(\R^n)$ are critical spaces (in the sense of scaling) for $n=4$ and $n=6$. In particular, it has been shown the local well-posedness in $L^2(\R^n)$, $1\leq n\leq4$ and in $H^1(\R^n)$, $1\leq n\leq6$. The method used to prove these results was the contraction argument combined with the well known Strichartz estimates (see, for instance, \cite{Cazenave} or \cite{Linares}). Under the assumption that there exists $c\in\R\backslash\{0\}$ such that $\lambda=c\overline{\mu}$ the global well-posedness in $L^2(\R^n)$, $1\leq n\leq3$ and in $H^1(\R^n)$, $1\leq n\leq3$, were also established. This condition on the parameters is necessary in order to obtain the conservation of the charge and the energy, which in turn imply a priori estimates. 
Since $L^2(\R^4)$ is a critical space, the global existence for $n=4$ requires an additional assumption on the initial data.  To be more precise on the results,  here and throughout the paper we assume:
\begin{equation}\label{conditionlambmu}
\mbox{There exist a constant $c\in \R\backslash\{0\}$ such that $\lambda=c\overline{\mu}$.} 
\end{equation}
By introducing the change of variables
 \begin{equation}\label{chagevariable}
\tilde{u}(x,t)=\sqrt{\frac{c}{2}}|\mu|u\left(\sqrt{\frac{1}{2m}}x,t\right),\;\;\;\;\;\;\;\tilde{v}(x,t)=-\frac{\lambda}{2}v\left(\sqrt{\frac{1}{2m}}x,t\right),
\end{equation}
we deduce that, after dropping the tildes,  \eqref{system1} is equivalent to
\begin{equation}\label{system3}
\begin{cases}
\displaystyle i\partial_{t}u+\Delta u=-2v\overline{u},\\
\displaystyle i\partial_{t}v+\kappa\Delta v=- u^{2},\\
(u(x,0),v(x,0))=(u_0,v_0),
\end{cases}
\end{equation}
with $\kappa=m/M$. So in what follows we will be concerned with system \eqref{system3} instead of \eqref{system1}.

It is not difficult to see that \eqref{system3} conserves the charge and the energy, which are given, respectively, by
\begin{equation}\label{conservationcharge1}
Q(u(t),v(t))=\|u(t)\|_{L^2}^{2}+2\|v(t)\|_{L^2}^{2}
\end{equation}
and
\begin{equation}\label{conservationenergia}
E(u(t),v(t))=\|\nabla u(t)\|_{L^2}^2+\kappa\|\nabla v(t)\|_{L^2}^2-2\;\mbox{Re} \big(v(t),u^{2}(t)\big)_{L^2}.
\end{equation}

With this terminology in hand it was established in \cite{Hayashi} that if $Q(u_0,v_0)<Q(\phi_{0},\psi_{0})$, where $(\phi_{0},\psi_{0})$ is  any \textit{ground state} associated with \eqref{system3} (see Section \ref{section2} for details) then the corresponding solution is global in $H^1(\R^4)$. This result is similar to the classical one established by Weinstein \cite{Weinstein} for the mass-critical NLS equation. To prove this result it was necessary to show a vectorial Gagliardo-Nirenberg-type inequality. In addition, under the mass resonance condition $\ka=1/2$ is was shown that this result is sharp in the sense that there exists a particular initial data $(u_0^*,v_0^*)$ such that $Q(u_0^*,v_0^*)=Q(\phi_{0},\psi_{0})$ but the corresponding local solution blows up in finite time. It is to be observed that
the existence of ground state solutions was established in any dimension $1\leq n\leq5$.
The used methods  were based on the concentration-compactness argument in dimension $n=1$ and on the Strauss' compactness lemma in dimension $2\leq n\leq5$ (see \cite{Strauss}).

We point out that, still under the condition $\ka=1/2$, if $E(u_0,v_0)<0$ (or $E(u_0,v_0)=0$ and $(u_0,v_0)$ has negative momentum) then the local solution also blows up in finite time in dimension $4\leq n\leq 6$ (see Theorem 3.12 in \cite{Hayashi}).

At this point we observe that in \cite{Hayashi} nothing was said concerning the global well-posedness in $H^1(\R^5)$. To the best of our knowledge there are no results in this direction in the current literature. 
 So, our main purpose in this notes is to obtain a sufficient sharp  condition for global well-posedness in $H^1(\R^5)$; this is the $L^2$ supercritical and $H^1$   subcritical case.

Before stating our results, let us introduce the functional
\begin{equation}\label{FK}
K(u,v)=\|\nabla u\|_{L^2}^{2}+\kappa\|\nabla v\|_{L^2}^{2}.
\end{equation}
Our main theorems reads as follows.

\begin{teore}[Sufficient condition for global existence]\label{thm:globalexistencecondn=5}
	Assume $(u_{0},v_{0})\in H^{1}(\R^5)\times H^{1}(\R^5)$ and let $(u(t),v(t))$ be the corresponding solution of \eqref{system3}, defined in the maximal interval of existence, say, $I$. Assume that  
	\begin{equation}\label{conditionsharp1}
	E(u_{0},v_{0})Q(u_{0},v_{0})<E(\phi,\psi)Q(\phi,\psi).
	\end{equation}
		If 
	\begin{equation}\label{conditionsharp2}
	K(u_{0},v_{0})Q(u_{0},v_{0})<K(\phi,\psi)Q(\phi,\psi),
	\end{equation}
	then
	\begin{equation*}
	K(u(t),v(t))Q(u_{0},v_{0})<K(\phi,\psi)Q(\phi,\psi),\;\;\;\;\;\forall t\in I.
	\end{equation*}
	In particular the initial-value problem (\ref{system3}) is globally well-posed in $H^{1}(\R^5)\times H^{1}(\R^5)$.
\end{teore}

Here,  $(\phi,\psi)$ is any ground state solution for (\ref{system4}) with $\omega =1$ (see Section \ref{section2}).
Under the mass resonance condition $\ka=1/2$,  next theorem shows that \eqref{conditionsharp2} is a sharp condition for the global existence.

\begin{teore}[Existence of blow-up solutions]\label{thm:sharpglobalexistencecondn=5} Let  $(\phi,\psi)$ be a ground state solution for (\ref{system4}) with $\omega=1$. Assume  $\kappa=1/2$ and $(u_{0},v_{0})\in H^{1}(\R^5)\times H^{1}(\R^5)$ and let $I$ be the maximal time interval of existence  of the solution $(u(t),v(t))$. Suppose
	\begin{equation}\label{energycondblowup}
	E(u_{0},v_{0})Q(u_{0},v_{0})<E(\phi,\psi)Q(\phi,\psi),
	\end{equation}
		and
	\begin{equation}\label{gradientcondblowup}
	K(u_{0},v_{0})Q(u_{0},v_{0})>K(\phi,\psi)Q(\phi,\psi).
	\end{equation}
If $(xu_{0},xv_{0})\in L^{2}(\R^5)\times L^{2}(\R^5)$ or $u_0,v_0$ are radial functions, then  $I$ is finite, which means to say that $(u(t),v(t))$ blows up in finite time.
\end{teore}

Theorems \ref{thm:globalexistencecondn=5} and \ref{thm:sharpglobalexistencecondn=5} are inspired in \cite{Holmer1}, where similar results were obtained for the classical Schr\"odinger equation.

This paper is organized as follows. In section \ref{section2} we state preliminary results concerned with the  ground state solutions. In particular, their  existence is shown via minimization of Weinstein's functional. We also establish a sharp Gagliardo-Nirenberg inequality.  In section \ref{section3}, we prove Theorems \ref{thm:globalexistencecondn=5} and \ref{thm:sharpglobalexistencecondn=5}.

\section{ground states and their properties}\label{section2}

In this section we will introduce the main tools to prove  Theorems \ref{thm:globalexistencecondn=5} and \ref{thm:sharpglobalexistencecondn=5}. In general, we use the standard notation in the theory of partial differential equations.

\subsection{Preliminary results}
Let us start by introducing the notion of ground states. First of all, we recall that a standing wave solution for \eqref{system3} is a solution of the form
$$
\left({u},{v}\right)=\left(e^{i\omega t}\phi(x),e^{2i\omega t} \psi(x)\right),
$$
where $\omega>0$ is a real parameter and $\phi,\psi$ are real-valued functions, which may depend on $\omega$, with a suitable decay at infinity.
By replacing this ansatz in \eqref{system3}, we face the elliptic system
\begin{equation}\label{system4}
\begin{cases}
\displaystyle -\Delta\phi+\omega \phi =2 \phi\psi,\\
\displaystyle -\kappa\Delta\psi+2\omega \psi= \phi^{2},
\end{cases}
\end{equation}
where as before $\ka=m/M$. In what follows, $\ka$ will be thought as any positive real constant.

\begin{defi}
A pair $(\phi,\psi)\in H^1(\R^n)\times H^1(\R^n)$ is called a solution (or a \textit{weak solution}) of \eqref{system4} if
$$
\int \left(\nabla \phi\cdot \nabla f +\omega \phi f\right)dx=2\int \phi\psi fdx,
$$
$$
\int \left(\ka\nabla \psi\cdot \nabla g +2\omega \psi g\right)dx=\int \phi^2 gdx,
$$
for any $f,g\in C_0^{\infty}(\R^n)$.
\end{defi}

As usual, from the standard elliptic regularity theory (see, for instance, \cite{Cazenave} or \cite{Evans}), weak solutions  are indeed smooth and  satisfy \eqref{system4} in the usual sense. The so called ground state solutions are usually obtained as minimizers of some functional connecting its critical points with the solutions of \eqref{system4}. The functional of interest here is defined by
\begin{equation*}
\begin{split}
I_{\omega}(\phi,\psi)&=\frac{1}{2}\left(\|\nabla \phi\|_{L^2}^{2}+\kappa\|\nabla \psi\|_{L^2}^{2}\right)+\frac{\omega}{2}\left(\| \phi\|_{L^2}^{2}+2\| \psi\|_{L^2}^{2}\right)-\int\phi^2\psi dx\\
&\equiv \frac{1}{2}E(\phi,\psi)+\frac{\omega}{2}Q(\phi,\psi).
\end{split}
\end{equation*}
In particular, it is easily seen that  $(\phi,\psi)$ is a solution of \eqref{system4} if and only if it is a critical point of $I_\omega$. As we said, among all solutions of \eqref{system3} stand out the so called \textit{ground states} we shall  define next.

\begin{defi}
A solution $(\phi_0,\psi_0)\in H^1(\R^n)\times H^1(\R^n)$ of \eqref{system4} is called a ground state if
$$
I_\omega(\phi_0,\psi_0)=\inf\left\{I_{\omega}(\phi,\psi); \; (\phi,\psi)\in \mathcal{C}_{\omega} \right\},
$$
where 
$$
\mathcal{C}_{\omega}:=\left\{ (\phi,\psi)\in H^1(\R^n)\times H^1(\R^n); \; (\phi,\psi)\; \mbox{is a nontrivial critical point of $I_\omega$} \right\}.
$$
The set of ground states for \eqref{system4} will be denoted by $\mathcal{G}_{\omega}$.
\end{defi}

The existence of ground states for \eqref{system4}, in dimension $1\leq n\leq5$, was already obtained in \cite{Hayashi} by minimizing the functional
$$
R_\omega(\phi,\psi)=\frac{K(\phi,\psi)+\omega Q(\phi,\psi)}{P(\phi,\psi)^{2/3}}
$$
on $\mathcal{P}$, where
	\begin{equation}
P(\phi,\psi):=\int\phi^2{\psi}\; dx\label{FP}
\end{equation}
and 
	$$
\mathcal{P}:=\left\{(\phi,\psi)\in H^{1}(\R^n)\times H^{1}(\R^n)\backslash\{(0,0)\};\int\phi^2\psi\; dx>0\right\}.
$$
 However, here we will present a slight different proof  based only on Strauss' compactness lemma. Indeed, in \cite{Hayashi} the authors used the compactness of the embedding $H^1_r(\R^n)\hookrightarrow L^3(\R^n)$, $2\leq n\leq5$ to obtain that any minimizing sequence converges to a ground state (up to scaling). Here, $H^1_r$ denotes the space of radially symmetric $H^1$ functions. Due to the lack of the above mentioned compactness in dimension $n=1$, they employed a concentration-compactness argument on the functional $I_\omega=\frac{1}{2}(K+\omega Q)-P$ to obtain the ground states in this case.

 In our approach, the ground states will be obtained as minimizers of a Weinstein-type functional. The advantage in using this approach is that, as a byproduct, it also yields  the best constant in a vectorial Gagliardo-Nirenberg inequality (see Corollary \ref{corollarybestconstant} below). This method was used in \cite{Hayashi} only for $n=4$. 
 
 Although we are mainly concerned with dimension $n=5$, we prove the existence of ground states for $1\leq n\leq5$, because it does not demand extra efforts. In fact, instead of using two different approaches for $n=1$ and $2\leq n\leq 5$, we use the compactness of the embedding $H^1_{rd}(\R^n)\hookrightarrow L^3(\R^n)$, which holds in any dimension $1\leq n\leq 5$ (see, for instance, \cite[Proposition 1.7.1]{Cazenave}). In particular, this simplifies the proof of the existence of ground states in dimension $n=1$. Here  $H^1_{rd}$ denotes the space of radially symmetric and nonincreasing $H^1$ functions.

We start with the following properties.

\begin{lem}\label{identitiesfunctionals}
	Let $(\phi,\psi)$ be a  solution of (\ref{system4}). Then,
	\begin{equation}
	P(\phi,\psi)=2I_{\omega}(\phi,\psi),\label{b}
	\end{equation}\begin{equation}
	K(\phi,\psi)=nI_{\omega}(\phi,\psi),\label{d}
	\end{equation}
	\begin{equation}
\omega	Q(\phi,\psi)=(6-n)I_{\omega}(\phi,\psi).\label{e}
	\end{equation}
\end{lem}
\begin{proof}
	See Theorem 4.1 in \cite{Hayashi}.	
\end{proof}

\begin{obs}\label{obsnorm}
In view of Lemma \ref{identitiesfunctionals} some simple remarks are in order:
	\begin{enumerate}
		\item From (\ref{e}) we  conclude that  a  solution $({\phi},{\psi})$ of (\ref{system4}) is a ground state if  and only if  the charge $Q({\phi},{\psi})$ is minimal.
		\item Ground states do not exist if $(6-n)\omega\leq0$. So, since we are assuming $\omega>0$, ground states do not exist if $n\geq6$.
		\item If $({\phi},{\psi})$ is a solution of (\ref{system4}) then (\ref{b}) and (\ref{d}) imply that $P(\phi,\psi)>0$; this means that $\mathcal{C}_{\omega}\subset \mathcal{P}$.
	\end{enumerate}
\end{obs}

Next we introduce the Weinstein functional
\begin{equation}\label{functionalJgeral2}
J(\phi,\psi)=\frac{\left(\| \phi\|_{L^2}^{2}+2\| \psi\|_{L^2}^{2}\right)^{\frac{3}{2}-\frac{n}{4}}\left(\|\nabla \phi\|_{L^2}^{2}+\kappa\|\nabla \psi\|_{L^2}^{2}\right)^{\frac{n}{4}}}{\int\phi^2\psi\; dx}\equiv 
\frac{Q(\phi,\psi)^{\frac{3}{2}-\frac{n}{4}}K(\phi,\psi)^{\frac{n}{4}}}{P(\phi,\psi)}.
\end{equation}
Note that if $(\phi,\psi)$ is a  solution of \eqref{system4} then, in view of Lemma \ref{identitiesfunctionals} and Remark \ref{obsnorm}, $J(\phi,\psi)$ is well-defined. 
In the following, we are going to use Lemma \ref{identitiesfunctionals} in order to get a relation between functionals $J$ and $I_{\omega}$. More precisely, we have.

\begin{lem}\label{lemma3}
	Let $(\phi,\psi)$ be a nontrivial solution of (\ref{system4}) then
	\begin{eqnarray}\label{relationJand Iomega}
	J(\phi,\psi)=\frac{n^{\frac{n}{4}}}{2}\left(\frac{6-n}{\omega}\right)^{\frac{3}{2}-\frac{n}{4}}I_{\omega}(\phi,\psi)^{\frac{1}{2}}.
	\end{eqnarray}
\end{lem}
\begin{proof}
	Combining expressions  (\ref{b}), (\ref{d}) and (\ref{e}) in Lemma \ref{identitiesfunctionals} and the definition of $J$,   we have
	\begin{eqnarray*}
		J(\phi,\psi)&=&\frac{Q(\phi,\psi)^{\frac{3}{2}-\frac{n}{4}}K(\phi,\psi)^{\frac{n}{4}}}{P(\phi,\psi)}\\
		&=&\frac{\left(\frac{6-n}{\omega}\right)^{\frac{3}{2}-\frac{n}{4}}I_{\omega}(\phi,\psi)^{\frac{3}{2}-\frac{n}{4}}n^{\frac{n}{4}}I_{\omega}(\phi,\psi)^{\frac{n}{4}}}{2I_{\omega}(\phi,\psi)}\\
		&=&\frac{n^{\frac{n}{4}}}{2}\left(\frac{6-n}{\omega}\right)^{\frac{3}{2}-\frac{n}{4}}I_{\omega}(\phi,\psi)^{\frac{1}{2}}.
	\end{eqnarray*}
The proof of the lemma is thus completed.	
\end{proof}

As an immediate consequence, we obtain the following.

\begin{coro}\label{corolario6.1}
	A nontrivial solution $(\phi,\psi)$ of (\ref{system4})   is a ground state if and only if it is a minimizer of $J$.
\end{coro}

In view of Corollary \ref{corolario6.1}, the idea to obtain the ground states is to minimize $J$ on the set $\mathcal{P}$.

\subsection{Existence of ground states}

Before proceeding we note that if we know a ground state for $\omega=1$ then we know a ground state for any $\omega>0$ (see Proposition  4.3 in \cite{Hayashi}). In fact, if
$(\phi_{1},\psi_{1})$ is a ground state for (\ref{system4}) with $\omega=1$, then	$$(\phi_{\omega}(x),\psi_{\omega}(x))=(\omega\phi_{1}(\sqrt{\omega}x),\omega\psi_{1}(\sqrt{\omega}x)),$$	is a ground state for (\ref{system4}), with $\omega>0$.

Our main theorem concerning ground states is the following.

\begin{teore}[Existence of ground states]\label{thm:existenceGSJgeral}
For  $1\leq n\leq 5$, the infimum
\begin{equation}\label{alph1def}
\alpha_1=\inf\limits_{(\phi,\psi)\in \mathcal{P}}J(\phi,\psi)
\end{equation}
is attained at a pair the functions $(\phi_{0},\psi_{0})\in \mathcal{P} $ such that
\begin{itemize}
		\item[(i)] $\phi_{0}$ and $\psi_{0}$ are non-negative and radially symmetric;
		
		\item[(ii)]   There exist $t_{0}>0$ and $l_{0}>0$ such that $(\phi,\psi)=(t_{0}\delta_{l_{0}}\phi_{0},t_{0}\delta_{l_{0}}\psi_{0})$ is a positive ground state  of \eqref{system4} with  $\omega=1$, where $(\delta_lf)(x)=f(x/l)$;
		
		\item[(iii)] If $(\tilde{\phi},\tilde{\psi})$ is any ground state of (\ref{system4}), with $\omega=1$, then
		\begin{equation}\label{inffunctionalJ}
		\alpha_{1}=\frac{n^{\frac{n}{4}}}{2}\left(6-n \right)^{1-\frac{n}{4}}Q(\tilde{\phi},\tilde{\psi})^{1/2}.
		\end{equation}
		
	\end{itemize}
\end{teore}

\begin{obs}
If $n=4$ then the constant $\alpha_{1}$ in \eqref{inffunctionalJ} reduces to $\alpha_{1}=2Q(\tilde{\phi},\tilde{\psi})^{1/2}$, which  is exactly the one obtained in \cite[Theorem 5.1]{Hayashi}.
\end{obs}

Below we will prove Theorem \ref{thm:existenceGSJgeral}. To begin with, we show the following.

\begin{lem}\label{lemmaalpha1}
If $\alpha_{1}$ is defined as in \eqref{alph1def}  then $\alpha_{1}>0$.
\end{lem}
\begin{proof}
First we recall  the Gagliardo-Nirenberg inequality $\|\phi\|_{L^3}\leq C\|\nabla \phi\|_{L^2}^{\frac{n}{6}}\| \phi\|_{L^2}^{1-\frac{n}{6}}$. Thus, from H\"older's inequality,
\[
\begin{split}
	P(\phi,\psi)&\leq\|\phi\|_{L^3}^{2}\|\psi\|_{L^3}\\
&\leq	C^3\Big(\|\nabla \phi\|_{L^2}^{\frac{n}{6}}\| \phi\|_{L^2}^{1-\frac{n}{6}}\Big)^{2}\|\nabla \psi\|_{L^2}^{\frac{n}{6}}\| \psi\|_{L^2}^{1-\frac{n}{6}}\\
&=C^3\kappa^{-\frac{n}{12}}2^{\frac{n}{12}-\frac{1}{2}}Q(\phi,\psi)^{\frac{3}{2}-\frac{n}{4}}K(\phi,\psi)^{\frac{n}{4}}.
\end{split}
\]
	Then, from the definition of  $\alpha_{1}$,
$$
0<C^{-3}\kappa^{\frac{n}{12}}2^{-\frac{n}{12}+\frac{1}{2}}\leq \alpha_{1},
$$
which yields the desired.
\end{proof}

In the sequel, given any non-negative function $\phi\in H^1(\R^n)$ we denote by $\phi^*$ its symmetric-decreasing rearrangement (see, for instance, \cite{Leoni}). Also, for any $l>0$,  $(\delta_{l}f)(x)=f\left(x/l\right)$.

\begin{lem}\label{scalling} Assume  $a,l>0$ and $(\phi,\psi)\in H^1(\R^n)\times H^1(\R^n)$. Then  the following properties hold:
\begin{itemize}
	\item[(i)] $Q(a\delta_l\phi,a\delta_l\psi)=a^{2}l^{n}Q(\phi,\psi), \qquad  Q'(a\delta_l\phi,a\delta_l\psi)(u,v)=al^nQ'(\phi,\psi)(\delta_{l^{-1}}u,\delta_{l^{-1}}v)$;
	\item[(ii)] $ K(a\delta_l\phi,a\delta_l\psi)=a^{2}l^{n-2}K(\phi,\psi), \quad K'(a\delta_l\phi,a\delta_l\psi)(u,v)=al^{n-2}K'(\phi,\psi)(\delta_{l^{-1}}u,\delta_{l^{-1}}v)$;
	\item[(iii)] $P(a\delta_l\phi,a\delta_l\psi)=a^{3}l^nP(\phi,\psi), \qquad  P'(a\delta_l\phi,a\delta_l\psi)(u,v)=a^{2}l^nP'(\phi,\psi)(\delta_{l^{-1}}u,\delta_{l^{-1}}v)$.\\
	In addition, if $\phi$ and $\psi$ are non-negative then
	\item[(iv)] $Q(\phi^{*},\psi^{*})=Q(\phi,\psi)$;
	\item[(v)] $K(\phi^{*},\psi^{*})\leq K(\phi,\psi)$;
	\item[(vi)] $P(\phi^{*},\psi^{*})\geq P(\phi,\psi)$.
\end{itemize}
\end{lem}
\begin{proof}
The proofs are simple calculations. For parts (iv), (v), and (vi) see, for instance, Chapters 6 and 16 in \cite{Leoni}.
\end{proof}

The following lemma establishes some properties of the functional $J$, under the transformations of scaling, dilation, and symmetrization.

\begin{lem}\label{lemma1}
If $a,l>0$ and $(\phi,\psi)\in\mathcal{P}$, then
\begin{enumerate}
\item[(i)] $J(a\delta_{l}\phi,a\delta_{l}\psi)=J(\phi,\psi)$;
\item[(ii)] $J(|\phi|,|\psi|)\leq J(\phi,\psi)$;
\item[(iii)] $J'(a\delta_{l}\phi,a\delta_{l}\psi)(u,v)=a^{-1}J'(\phi,\psi)(\delta_{l^{-1}}u,\delta_{l^{-1}}v)$.\\
	In addition, if $\phi$ and $\psi$ are non-negative then
\item[(iv)] $J(\phi^{*},\psi^{*})\leq J(\phi,\psi)$. 
\end{enumerate}
\end{lem}
\begin{proof}
The proofs are immediate consequences of Lemma \ref{scalling}.
\end{proof}

Now we are able to proof Theorem \ref{thm:existenceGSJgeral}. We will follow the arguments presented in references \cite{Hayashi} and \cite{Weinstein}.

\begin{proof}[Proof of Theorem \ref{thm:existenceGSJgeral}]

Let $\{(\phi_{j},\psi_{j})\}\subset \mathcal{P} $ be a minimizing sequence for $J$, i.e., 
$$
\lim_{j\to \infty}J(\phi_{j},\psi_{j})=\alpha_{1}.
$$
In view of Lemma \ref{lemma1} we have $J(|\phi_j|,|\psi_j|)\leq J(\phi_j,\psi_j)$. So, we may assume that $\phi_{j},\psi_{j}$ are non-negative. In addition, also from Lemma \ref{lemma1},  $J(\phi_{j}^{*},\psi_{j}^{*})\leq J(\phi_{j},\psi_{j})$; thus, we also may assume that  $\phi_{j},\psi_{j}$ are radially symmetric and nonincreasing functions in $H^{1}$.
Define 
$$\tilde{\phi_{j}}=t_{j}\delta_{l_{j}}\phi_{j}\;\;\;\;\;\;\;\;\mbox{and}\;\;\;\;\;\;\;\;\tilde{\psi_{j}}=t_{j}\delta_{l_{j}}\psi_{j},$$
where
$$
t_{j}=\frac{Q(\phi_{j},\psi_{j})^{\frac{n}{4}-\frac{1}{2}}}{K(\phi_{j},\psi_{j})^{\frac{n}{4}}}\;\;\;\;\;\;\;\;\mbox{and}\;\;\;\;\;\;\;\;l_{j}=\frac{K(\phi_{j},\psi_{j})^{\frac{1}{2}}}{Q(\phi_{j},\psi_{j})^{\frac{1}{2}}}.
$$
An application of Lemma \ref{scalling},    with $a=t_{j}$ and $l=l_{j}$, gives
\begin{equation}\label{equalnorm}
K(\tilde{\phi_{j}},\tilde{\psi_{j}})=Q(\tilde{\phi_{j}},\tilde{\psi_{j}})=1.
\end{equation}
Hence, 
\begin{equation}\label{convergenceP3}
\frac{1}{P(\tilde{\phi_{j}},\tilde{\psi_{j}})}=\frac{Q(\tilde{\phi_{j}},\tilde{\psi_{j}})^{\frac{3}{2}-\frac{n}{4}}K(\tilde{\phi_{j}},\tilde{\psi_{j}})^{\frac{n}{4}}}{P(\tilde{\phi_{j}},\tilde{\psi_{j}})}=J(\tilde{\phi_{j}},\tilde{\psi_{j}})=J(\phi_{j},\psi_{j})\to \alpha_{1}>0.
\end{equation}

Recall that $H_{rd}^{1}(\R^n)$ denotes the closed subspace of $H^{1}(\R^n)$ composed by radially symmetric and nonincreasing functions. It follows from \eqref{equalnorm} that sequence $(\tilde{\phi_{j}},\tilde{\psi_{j}})$ is bounded in the space $H_{rd}^{1}(\R^n)\times H_{rd}^{1}(\R^n)$. Consequently, there exist a subsequence, still denoted by $(\tilde{\phi_{j}},\tilde{\psi_{j}})$, and functions  $\phi_{0},\psi_{0}\in  H_{rd}^{1}(\R^n)$  such that 
 \begin{equation*}
 \tilde{\phi_{j}}\rightharpoonup \phi_{0},\;\;\;\;\;\;\;\;\tilde{\psi_{j}}\rightharpoonup \psi_{0}, \;\;\;\;\;\;\mbox{in\;\; $H^{1}(\R^n)$}.
 \end{equation*}
By recalling the compactness of the embedding $ H_{rd}^{1}(\R^n)\hookrightarrow  L^{3}(\R^n)$, $1\leq n\leq 5$, we see that
 $(\tilde{\phi_{j}},\tilde{\psi_{j}})\to (\phi_{0},\psi_{0})$, in $L^3\times L^3$ and \textit{almost everywhere}. This immediately implies that $\phi_0$ and $\psi_0$ are non-negative and
 $$
 \lim_{j\to \infty}P(\tilde{\phi_{j}},\tilde{\psi_{j}})=P(\phi_{0},\psi_{0}).
 $$
 Therefore by (\ref{convergenceP3}) we get
\begin{equation}\label{relationPandalpha2}
P(\phi_{0},\psi_{0})=\lim_{j\to \infty}P(\tilde{\phi_{j}},\tilde{\psi_{j}})=\alpha_{1}^{-1}>0,
\end{equation}
which means that $(\phi_{0},\psi_{0})\in \mathcal{P}$.

On the other hand, the lower semi-continuity of the weak convergence gives
\[
Q(\phi_{0},\psi_{0})
\leq\liminf_{j\to\infty}Q(\tilde{\phi_{j}},\tilde{\psi_{j}})=1
\]
and
\[
K(\phi_{0},\psi_{0})
\leq\liminf_{j\to\infty}K(\tilde{\phi_{j}},\tilde{\psi_{j}})=1.
\]

Therefore,  the definitions of $\alpha_{1}$ and $J$ and \eqref{relationPandalpha2}  yield
\begin{equation}\label{inequJKandQ2}
\alpha_{1}\leq  J(\phi_{0},\psi_{0})=\frac{Q(\phi_{0},\psi_{0})^{\frac{3}{2}-\frac{n}{4}}K(\phi_{0},\psi_{0})^{\frac{n}{4}}}{P(\phi_{0},\psi_{0})}\leq \frac{1}{P(\phi_{0},\psi_{0})}=\alpha_{1}.
\end{equation}
From \eqref{inequJKandQ2} we then conclude that
\begin{equation}\label{alphaJ}
J(\phi_{0},\psi_{0})=\alpha_{1}
\end{equation}
and
\begin{equation}\label{equal1}
K(\phi_{0},\psi_{0})=Q(\phi_{0},\psi_{0})=1.
\end{equation}
This last assertion also implies that $\tilde{\phi_{j}}\to \phi_{0}$, $\tilde{\psi_{j}}\to \psi_{0}$ strongly in $H^{1}$. Part (i) of the theorem is thus established.

 For part (ii), we start by noting that for any $(u,v)\in H^1\times H^1$ and $t$ sufficiently small, $(\phi_{0}+tu,\psi_{0}+tv)\in \mathcal{P}$. Thus, since $(\phi_0,\psi_0)$ is a minimizer of $J$ on $\mathcal{P}$,
\begin{equation*}
\left.\frac{d}{dt}\right|_{t=0}J(\phi_{0}+tu,\psi_{0}+tv)=0.
\end{equation*}
Using  Lemma \ref{scalling}, this is equivalent to
\[
\begin{split}
\frac{Q(\phi_{0},\psi_{0})^{\frac{3}{2}-\frac{
			n}{4}}K(\phi_0,\psi_0)^{\frac{n}{4}}}{P(\phi_{0},\psi_{0})}\left[\frac{n}{4}\frac{K'(\phi_{0},\psi_{0})(u,v)}{K(\phi_0,\psi_0)}+\left(\frac{3}{2}-\frac{n}{4}\right)\frac{1}{Q(\phi_{0},\psi_{0})}Q'(\phi_{0},\psi_{0})(u,v)\right]\\
=\frac{Q(\phi_{0},\psi_{0})^{\frac{3}{2}-\frac{n}{4}}K(\phi_{0},\psi_{0})^{\frac{n}{4}}}{P(\phi_{0},\psi_{0})^{2}}P'(\phi_{0},\psi_{0})(u,v).
\end{split}
\]
In view of \eqref{relationPandalpha2} and \eqref{equal1}, this yields, for any $(u,v)\in H^1\times H^1$,
\begin{equation}
K'(\phi_{0},\psi_{0})(u,v)+\frac{6-n}{n}Q'(\phi_{0},\psi_{0})(u,v)=\frac{4\alpha_{1}}{n}P'(\phi_{0},\psi_{0})(u,v).\label{relationKQandPgeral}
\end{equation}

Next, define  $(\phi,\psi)=(t_{0}\delta_{l_{0}}\phi_{0},t_{0}\delta_{l_{0}}\psi_{0})$ with 
$$t_{0}=\displaystyle \frac{2 \alpha_{1} }{6-n}\;\;\;\;\;\;\;\;\;\mbox{and}\;\;\;\;\;\;\;\;l_{0}=\left(\frac{6-n}{ n}\right)^{1/2}.$$
We claim that  $(\phi,\psi)$ is a solution of (\ref{system4}) with $\omega=1$, that is, $(\phi,\psi)$ is a critical point of $I_1$. To see this, we note that for any $u,v\in H^{1}$, in view of Lemma \ref{scalling},
\[
\begin{split}
&I'_{1}(\phi,\psi)(u,v)\\
&=\frac{1}{2}\left[K'(t_{0}\delta_{l_{0}}\phi_{0},t_{0}\delta_{l_{0}}\psi_{0})(u,v)+ Q'(t_{0}\delta_{l_{0}}\phi_{0},t_{0}\delta_{l_{0}}\psi_{0})(u,v)\right]-P'(t_{0}\delta_{l_{0}}\phi_{0},t_{0}\delta_{l_{0}}\psi_{0})(u,v)\\
&=\frac{t_{0}l_{0}^{n-2}}{2}\left[K'(\phi_{0},\psi_{0})(\delta_{l_{0}^{-1}}u,\delta_{l_{0}^{-1}}v)+l_{0}^{2} Q'(\phi_{0},\psi_{0})(\delta_{l_{0}^{-1}}u,\delta_{l_{0}^{-1}}v)-2t_{0}l_{0}^{2}P'(\phi_{0},\psi_{0})(\delta_{l_{0}^{-1}}u,\delta_{l_{0}^{-1}}v)\right]\\
&=\frac{t_{0}l_{0}^{n-2}}{2}\bigg[K'(\phi_{0},\psi_{0})(\delta_{l_{0}^{-1}}u,\delta_{l_{0}^{-1}}v)+\frac{6-n}{n}Q'(\phi_{0},\psi_{0})(\delta_{l_{0}^{-1}}u,\delta_{l_{0}^{-1}}v)\\
& \qquad \qquad\qquad\qquad -\frac{4\alpha_{1} }{n}P'(\phi_{0},\psi_{0})(\delta_{l_{0}^{-1}}u,\delta_{l_{0}^{-1}}v)\bigg]\\
&=0, 
\end{split}
\]
where in the last line we used \eqref{relationKQandPgeral}.

Now from Lemmas \ref{lemma3} and \ref{lemma1} we have that $(\phi,\psi)$ is also a critical point of $J$ with $J(\phi,\psi)=J(\phi_0,\psi_0)$. Since $(\phi_{0},\psi_{0})$ is a minimizer of $J$, so is $(\phi,\psi)$. An application of Corollary \ref{corolario6.1} then gives that $(\phi,\psi)$ is a ground state of (\ref{system4}) with $\omega=1$.
Finally, the positiveness of $(\phi,\psi)$ follows from the maximum principle. This shows part (ii).

Next we will prove the relation (\ref{inffunctionalJ}). Indeed, if $(\phi,\psi)$ is as in part (ii), Lemmas \ref{lemma3} and \ref{identitiesfunctionals} imply, 
\[
\begin{split}
\alpha_1&=J(\phi,\psi)\\
&=\frac{n^{\frac{n}{4}}}{2}\left(6-n\right)^{\frac{3}{2}-\frac{n}{4}}I_{1}(\phi,\psi)^{\frac{1}{2}}\\
& =\frac{n^{\frac{n}{4}}}{2}\left(6-n \right)^{1-\frac{n}{4}}Q(\phi,\psi)^{\frac{1}{2}}.
\end{split}
\]
Therefore, if $(\tilde{\phi},\tilde{\psi})$ is any ground state of \eqref{system4}, with $\omega=1$,  Remark \ref{obsnorm} yields
$$
\alpha_{1}=\frac{n^{\frac{n}{4}}}{2}\left(6-n \right)^{1-\frac{n}{4}}Q(\tilde{\phi},\tilde{\psi})^{\frac{1}{2}}.
$$
The proof of the theorem is thus completed.
\end{proof}

From the proof of Lemma \ref{lemmaalpha1} we deduce the existence of a constant $C>0$ such that the Gagliardo-Nirenberg-type inequality
\begin{equation}\label{sharpgn}
P(u,v)\leq C Q(u,v)^{\frac{3}{2}-\frac{n}{4}}K(u,v)^{\frac{n}{4}}
\end{equation}
holds,  for any $(u,v)\in \mathcal{P}$. In view of Theorem \ref{thm:existenceGSJgeral} the sharp constant one can place in \eqref{sharpgn} is $\alpha_1^{-1}$. More precisely, we have.

\begin{coro}\label{corollarybestconstant}
Let $1\leq n\leq 5$. Then the inequality
$$
P(u,v)\leq C_{op} Q(u,v)^{\frac{3}{2}-\frac{n}{4}}K(u,v)^{\frac{n}{4}}
$$
holds,  for any $(u,v)\in \mathcal{P}$, with
\begin{equation*}
C_{op}=\frac{2\left(6-n\right)^{\frac{n}{4}-1}}{n^{\frac{n}{4}}}\frac{1}{Q({\phi},{\psi})^{\frac{1}{2}}},
\end{equation*}
where $({\phi},{\psi})$ is any ground state solution of (\ref{system4}) with $\omega=1$.
\end{coro}

\begin{obs}
When $n=4$, we recover the best constant obtained by the authors in \cite[Theorem 5.1]{Hayashi}.
\end{obs}

\subsection{Characterization of ground states}

Next we present a characterization of the minimizers of $J$. The result generalizes the one for $n=4$ present in   \cite[Theorem 5.1]{Hayashi}.

\begin{teore}
Let $\alpha_{1}$ be defined as in \eqref{alph1def}. Then the set of minimizers of $J$ is characterized as
$$
\{(\phi,\psi)\in H^{1}\times H^{1};J(\phi,\psi)=\alpha_{1}\}=\{(t\delta_{l}\phi,t\delta_{l}\psi)\in H^{1}\times H^{1};\;t,l>0,(\phi,\psi)\in \mathcal{G}_{1}\}.
$$
\end{teore}
\begin{proof}
Consider the following sets $$A=\{(\phi,\psi)\in H^{1}\times H^{1};J(\phi,\psi)=\alpha_{1}\}$$
and 
$$B=\{(t\delta_{l}\phi,t\delta_{l}\psi)\in H^{1}\times H^{1};\;t,l>0,(\phi,\psi)\in \mathcal{G}_{1}\}
.
$$ 
If $(\phi_{0},\psi_{0})\in A$ then  $J(\phi_{0},\psi_{0})=\alpha_{1}$ and from Theorem \ref{thm:existenceGSJgeral}  there exist $t_{0},l_{0}>0$ such that $(\phi,\psi)=(t_{0}\delta_{l_{0}}\phi_{0},t_{0}\delta_{l_{0}}\psi_{0})$ is a ground state of (\ref{system4}). This means that
$$(\phi_{0},\psi_{0})=\left(\frac{1}{t_{0}}\delta_{l_{0}^{-1}}\phi,\frac{1}{t_{0}}\delta_{l_{0}^{-1}}\psi\right),
$$
with $(\phi,\psi)\in \mathcal{G}_{1}$ or,  equivalently, $(\phi_{0},\psi_{0})\in B$.

On the other hand, assume $(\phi_0,\psi_0)\in B$, that is,
 $(\phi_{0},\psi_{0})=(t\delta_{l}\phi,t\delta_{l}\psi)$ for some $t,l>0$ with  $(\phi,\psi)\in \mathcal{G}_{1}$. We must  proof that $J(\phi_{0},\psi_{0})=\alpha_{1}$. From Lemma \ref{lemma1} we have $J(\phi_0,\psi_0)=J(\phi,\psi)$.  But since $(\phi,\psi)\in \mathcal{G}_1$, Corollary \ref{corolario6.1} implies that $(\phi,\psi)$  is a minimizer of $J$, that is, $J(\phi,\psi)=\alpha_{1}$. Consequently, $J(\phi_0,\psi_0)=\alpha_{1}$ and $(\phi_0,\psi_0)\in A$.
\end{proof}

\section{Global existence versus blow up}\label{section3}

In this section, we prove Theorems \ref{thm:globalexistencecondn=5} and \ref{thm:sharpglobalexistencecondn=5}. Before proving the results, we need some previous tools. We start with two lemmas. Their proofs  can be found, for instance, in references  \cite{beg},  \cite{Esfahani}, and \cite{Pastor}.

\begin{lem}\label{supercritcalcase}
Let $I$ be an open interval with $0\in I$. Let $a\in \R$ and $b>0$. Define $\gamma=(bq)^{-\frac{1}{q-1}}$ and $f(r)=a-r+br^{q}$, for $r\geq 0$. Let $G(t)$ be a nonnegative continuous  function such that $f\circ G\geq 0$ on $I$. Assume that $a<\left(1-\frac{1}{q}\right)\gamma$.
\begin{enumerate}
\item[(i)] If $G(0)<\gamma$, then $G(t)<\gamma$, $\forall t\in I$.
\item[(ii)] If $G(0)>\gamma$, then $G(t)>\gamma$, $\forall t\in I$.
\end{enumerate}
\end{lem}

\begin{coro}\label{corosupercritcalcase}
Let $I$ be an open interval with $0\in I$. Let $a\in \R$ and $b>0$. Define $\gamma=(bq)^{-\frac{1}{q-1}}$ and $f(r)=a-r+br^{q}$, for $r\geq 0$. Let $G(t)$ be a nonnegative continuous  function such that $f\circ G\geq 0$ on $I$. Assume that $a<(1-\delta_{1})\left(1-\frac{1}{q}\right)\gamma$, for some small $\delta_{1}>0$.
If $G(0)>\gamma$, then there exists $\delta_{2}=\delta_{2}(\delta_{1})>0$ such that $G(t)>(1+\delta_{2})\gamma$, $\forall t\in I$.
\end{coro}

\subsection{Global existence in $H^{1}(\R^{5})$.}

 By using Corollary \ref{corollarybestconstant} and standard arguments one can show if $\|(u_0,v_0)\|_{H^1\times H^1}\leq \rho$, for some $\rho$ sufficiently small  then the corresponding solution of \eqref{system3} is global in $H^1(\R^5)\times H^1(\R^5)$ (see, for instance, \cite{Linares}). Particularly, if $\rho$ is small then the assumptions of Theorem \ref{thm:globalexistencecondn=5} hold. Thus, Theorem \ref{thm:globalexistencecondn=5} can be seen as an answer to the question of how small $\rho$ must be.
 
 The proof of Theorem \ref{thm:globalexistencecondn=5} will follow as an application of Lemma \ref{supercritcalcase}. To simplify it we first prove the following.

\begin{lem}\label{lemmauxiliar}
	Let $(u(t),v(t))$ be the solution of \eqref{system3} with initial data $(u_0,v_0)\in H^1(\R^5)\times H^1(\R^5)$.
 Define $G(t)=K(u(t),v(t))$, $a=E(u_{0},v_{0})$, $b=2C_{op}Q(u_{0},v_{0})^{\frac{1}{4}}$ and $\displaystyle q=\frac{5}{4}$. Then,
\begin{enumerate}
\item[(i)]  $f\circ G\geq 0$, where $f(r)=a-r+br^q$.
\item[(ii)] If $\gamma=(bq)^{-\frac{1}{q-1}}$ then
\begin{enumerate}
\item $\displaystyle a<\left(1-\frac{1}{q}\right)\gamma$ $\Longleftrightarrow$ $E(u_{0},v_{0})Q(u_{0},v_{0})<E(\phi,\psi)Q(\phi,\psi)$;
\vspace{0.3cm}
\item $G(0)<\gamma$ $\Longleftrightarrow$ $Q(u_{0},v_{0})K(u_{0},v_{0})<Q(\phi,\psi)K(\phi,\psi)$.
\end{enumerate}
\end{enumerate}
\end{lem}

\begin{proof}
For part (i), from the definition of the energy and Corollary \ref{corollarybestconstant}, with $n=5$, we have
\[
\begin{split}
K(u(t),v(t))&=E(u_0,v_0)+2\,{\rm Re}\left(v(t),u(t)^2\right)_{L^2}\\
&\leq E(u_0,v_0)+2P(|u(t)|^2,|v(t)|)\\
&\leq E(u_0,v_0)+2C_{op}Q(u(t),v(t))^{\frac{1}{4}}K(u(t),v(t))^{\frac{5}{4}}.
\end{split}
\]
The conservation of the charge then gives part (i).

For part (ii), we first observe that from Lemma \ref{identitiesfunctionals} with $n=5$ and $\omega=1$,
\begin{equation*}
Q(\phi,\psi)=\frac{1}{5}K(\phi,\psi) \qquad \mbox{and} \qquad P(\phi,\psi)=2Q(\phi,\psi).
\end{equation*}
Hence,
\begin{equation}\label{relationQPK}
E(\phi,\psi)=K(\phi,\psi)-2P(\phi,\psi)
=5Q(\phi,\psi)-4Q(\phi,\psi)
=Q(\phi,\psi).
\end{equation}
In addition, by definition,
\begin{equation}\label{gammadefe}
\gamma=5\frac{Q(\phi,\psi)^{2}}{Q(u_{0},v_{0})}.
\end{equation}
By combining \eqref{relationQPK}-\eqref{gammadefe}, part (ii) follows from straightforward calculations.
\end{proof}

Finally we are able to prove Theorem \ref{thm:globalexistencecondn=5}.

\begin{proof}[Proof of Theorem~\ref{thm:globalexistencecondn=5}]
From Lemma \ref{lemmauxiliar}, assumptions \eqref{conditionsharp1} and \eqref{conditionsharp2} are equivalent to $a<\left(1-\frac{1}{q}\right)\gamma$  and $G(0)<\gamma$. Thus from Lemma \ref{supercritcalcase} we conclude that  $G(t)<\gamma$, or equivalently,
$$
Q(u_{0},v_{0})K(u(t),v(t))<Q(\phi,\psi)K(\phi,\psi),\;\;\;\;\;\forall t\in I. 
$$
This combined with the conservation of the charge implies an a priori estimate for the solution in $H^1\times H^1$. Consequently, the solution can be extended globally-in-time and the proof is completed. 
\end{proof}

%\begin{obs}
%We note that if $n=4$ conditions (\ref{conditionsharp1}) and (\ref{conditionsharp2}) are the same and we recover the result in Theorem \ref{theoremsharpconditionsubcriticalcase} given in \cite{Hayashi}.
%\end{obs}

%%%%%%%%%%%%%%%%%%%%%%%%%%%%%%%%%%%%%%%%%%%%%%%%%%%%%%%%%%%%%%%%%%%%%%%%%%%%%
\subsection{Existence of blow-up solutions}
Here we will show Theorem \ref{thm:sharpglobalexistencecondn=5}.
Before starting with the proof itself, we need some preliminary results. 
First we recall the virial identities for system \eqref{system3}. The first one is used to prove the blow up under the assumption of finite variance.

\begin{lem}
Assume $1\leq n\leq 6$ and  $\kappa=1/2$. Let $(u(t),v(t))$ be the unique  solution of (\ref{system3}) with $(u_{0},v_{0})\in H^{1}(\R^{n})\times H^{1}(\R^{n})$ and  $(xu_{0},xv_{0})\in L^{2}(\R^n)\times L^{2}(\R^n)$. Then, as long as the solution exists,
\begin{equation}\label{virialidentity2}
\frac{d^2}{dt^2}Q(xu(t),xv(t))=2nE(u_{0},v_{0})+2(4-n)K(u(t),v(t)).
\end{equation}
\end{lem}
\begin{proof}
See \cite[Theorem 3.11]{Hayashi}. In particular, in  \cite[Theorem 3.8]{Hayashi} it was shown that, under our assumptions, the local solution also satisfies $(xu(t),xv(t))\in L^{2}(\R^n)\times L^{2}(\R^n)$.
\end{proof}

\begin{teore}\label{Viarialidenityradialcase} 
	Assume $1\leq n\leq 6$ and  $\kappa=1/2$. Let $(u(t),v(t))$ be the unique  solution of (\ref{system3}) with $(u_{0},v_{0})\in H^{1}(\R^{n})\times H^{1}(\R^{n})$. Let $\varphi \in C^{\infty}_{0}(\R^{n})$ and define
\begin{equation*}
V(t)=\frac{1}{2}\int \varphi(x)(|u|^{2}+2|v|^{2})dx.
\end{equation*}
Then,
\begin{equation*}
V'(t)=\mbox{Im}\int\nabla \varphi\cdot \nabla u \overline{u}\;dx+\mbox{Im}\int\nabla \varphi\cdot \nabla v \overline{v}\;dx
\end{equation*}
and
\begin{equation}\label{secondervgeralcase}
\begin{split}
V''(t)&=2\sum_{1\leq k,j\leq n}{\rm Re}\int\frac{\partial^{2}\varphi}{\partial x_{k}\partial x_{j}}\left[\partial x_{j}\overline{u}\partial x_{k}u+\frac{1}{2}\partial x_{j}\overline{v}\partial x_{k}v\right]dx\\
&\quad -\frac{1}{2}\int\Delta^{2}\varphi\left(|u|^{2}+\frac{1}{2}|v|^{2}\right)dx-{\rm Re}\int\Delta\varphi \overline{v}u^{2}\;dx.
\end{split}
\end{equation}
\end{teore}
\begin{proof}
The proof follows the ideas presented in Lemma 2.9 of \cite{Kavian}, where the virial identity was established for the classical Schr\"odinger equation. So we omit the details. An adapted version  for Schr\"odinger-type systems, can also be found in \cite{Pastor}.
\end{proof}

\begin{coro}\label{key}
	Under the assumptions of Theorem \ref{Viarialidenityradialcase},
if  $u,v$ and $\varphi$ are
 radially symmetric functions then we can write (\ref{secondervgeralcase}) as
 \begin{equation}\label{secondervradialcase}
 \begin{split}
V''(t)&=2\int \varphi''\left(|\nabla u|^{2}+\frac{1}{2}|\nabla v|^{2}\right)dx-\frac{1}{2}\int\Delta^{2}\varphi\left(|u|^{2}+\frac{1}{2}|v|^{2}\right)dx\\
&\quad-{\rm Re}\int\Delta\varphi \overline{v}u^{2}\;dx.
\end{split}
\end{equation}
\end{coro}
\begin{proof}
The proof follows immediately from Theorem \ref{Viarialidenityradialcase}.
\end{proof}

We will use Corollary \ref{key} with $\varphi$ replaced by the function $\chi_R$ given in next lemma.

\begin{lem}\label{lemafunctionchi} Let  $r=|x|$, $x\in \R^{n}$. Define
\begin{equation}\label{definitionchi}
 \chi(r)=\left\{\begin{array}{cc}
r^{2},&0\leq r\leq 1,\\
0,& r\geq 3,
\end{array}\right.
\end{equation}
 with $\chi''(r)\leq 2$, for any $r\geq 0$. For $R>0$, let $\chi_{R}(r)=R^{2}\chi\left(r/R\right)$.
 \begin{enumerate}
 \item[(i)] If $r\leq R$, then $\Delta\chi_{R}(r)=2n$ and $\Delta^{2}\chi_{R}(r)=0$;
 \item[(ii)] If $r\geq R$, then
 \begin{equation}\label{laplacianchiproperty}
 \Delta\chi_{R}(r)\leq C_{1},\;\;\;\; \;\;\;\;\;\;\Delta^{2}\chi_{R}(r)\leq \frac{C_{2}}{R^{2}},
 \end{equation}
 where $C_{1},C_{2}$ are constant depending only on  $n$.
 \end{enumerate}
\end{lem}
\begin{proof}
The lemma follows by a straightforward calculation.
\end{proof}

Finally, we recall a truncated version of the Gagliardo-Nirenberg inequality.

\begin{lem}\label{StraussLemaconsequence}
	If $u\in H^{1}(\R^{5})$ is a radially symmetric function, then
	\begin{equation}\label{StraussLemmaineq}
	\|u\|^{3}_{L^{3}(|x|\geq R)}\leq \frac{C}{R^{2}}\|u\|^{5/2}_{L^{2}(|x|\geq R)}\|\nabla u\|^{1/2}_{L^{2}(|x|\geq R)}
	\end{equation}
\end{lem}
\begin{proof}
	The proof is a consequence of Strauss' radial Lemma (See also \cite[page 323]{Ogawa}).
\end{proof}

Finally, we are in a position to prove the existence of blow-up solutions. We will follow the strategy presented in  \cite{Holmer2}, \cite{Pastor} and \cite{Ogawa}.

\begin{proof}[Proof of Theorem~\ref{thm:sharpglobalexistencecondn=5}]
Since we are assuming (\ref{energycondblowup}) we have the existence of  $\delta_{1}>0$ such that
\begin{equation}\label{conddelta1}
E(u_{0},v_{0})Q(u_{0},v_{0})<(1-\delta_{1})E(\phi,\psi)Q(\phi,\psi).
\end{equation}
Hence, as in the proof of Lemma \ref{lemmauxiliar}, we are in the assumptions of  Corollary \ref{corosupercritcalcase}. Consequently, there exists $\delta_{2}>0$ such that
\begin{equation}\label{conddelta2}
Q(u_{0},v_{0})K(u(t),v(t))>(1+\delta_{2})Q(\phi,\psi)K(\phi,\psi).
\end{equation}

We first assume that $(xu_{0},xv_{0})\in L^{2}(\R^5)\times L^{2}(\R^5)$. Then multiplying both side of (\ref{virialidentity2})  by $Q(u_{0},v_{0})$ we have
\begin{eqnarray*}
\left[\frac{d^2}{dt^2}Q(xu,xv)\right]Q(u_{0},v_{0})&=&10E(u_{0},v_{0})Q(u_{0},v_{0})-2K(u,v)Q(u_{0},v_{0})\\
&<&10(1-\delta_{1})E(\phi,\psi)Q(\phi,\psi)-2K(u,v)Q(u_{0},v_{0})\\
&<&10(1-\delta_{1})E(\phi,\psi)Q(\phi,\psi)-2(1+\delta_{2})Q(\phi,\psi)K(\phi,\psi)\\
&=&2(1-\delta_{1})K(\phi,\psi)Q(\phi,\psi)-2(1+\delta_{2})Q(\phi,\psi)K(\phi,\psi)\\
&=&-2(\delta_{1}+\delta_{2})Q(\phi,\psi)K(\phi,\psi)=-B,
\end{eqnarray*}
where we have used that $E(\phi,\psi)=K(\phi,\psi)/5$. Since $B$ is a positive constant, by using standard arguments, the last inequality is enough to show that $I$ must be finite.

Next, we assume that $u_0,v_0$ are radial functions. Because the linear and nonlinear parts in \eqref{system3} are invariant by rotations, it is easy to see that $u(t),v(t)$ are also radial functions.

 Let $\chi_{R}$ be defined as in Lemma \ref{lemafunctionchi}. The parameter $R$ is fixed at this moment but it will be chosen sufficiently large later. By taking $\varphi$ as $\chi_{R}$ in \eqref{secondervradialcase}, we obtain
 \begin{equation}\label{secodnderivativeVwithchi}
 \begin{split}
V''(t)&=2\int \chi_{R}''\left(|\nabla u|^{2}+\frac{1}{2}|\nabla v|^{2}\right)dx-\frac{1}{2}\int\Delta^{2}\chi_{R}\left(|u|^{2}+\frac{1}{2}|v|^{2}\right)dx\\
&\quad -\mbox{Re}\int\Delta\chi_{R} \overline{v}u^{2}\;dx.
\end{split}
\end{equation}
We will estimate each term on the right-hand side of \eqref{secodnderivativeVwithchi}. For the first one, by using the fact that $\chi_{R}''(r)\leq 2$, we have

\begin{equation}\label{estim1}
2\int \chi_{R}''\left(|\nabla u|^{2}+\frac{1}{2}|\nabla v|^{2}\right)dx\leq 4\int\left(|\nabla v|^{2}+\frac{1}{2}|\nabla u|^{2}\right)dx= 4K(u,v).
\end{equation}
For the second one  we use    Lemma \ref{lemafunctionchi} and the conservation of the charge to get
\begin{equation}\label{estim2}
\begin{split}
-\frac{1}{2}\int\Delta^{2}\chi_{R}\left(|u|^{2}+\frac{1}{2}|v|^{2}\right)dx&=-\frac{1}{2}\int\limits_{|x|\geq R}\Delta^{2}\chi_{R}\left(|u|^{2}+\frac{1}{2}|v|^{2}\right)dx\\
&\leq\frac{C_{2}}{R^{2}}\int\limits_{|x|\geq R}\left(|u|^{2}+\frac{1}{2}|v|^{2}\right)dx\\
&\leq\frac{C_{2}'}{R^{2}}\int\limits_{\R^{5}}\left(|u|^{2}+2|v|^{2}\right)dx\\
&=\frac{C_{2}'}{R^{2}}Q(u_{0},v_{0}),
\end{split}
\end{equation}
where $C_2'$ is a positive constant.

Finally, in view of Lemma \ref{lemafunctionchi}, the last term in  (\ref{secodnderivativeVwithchi}) is estimated by
\begin{equation}\label{estim3}
\begin{split}
-\mbox{Re}\int\Delta\chi_{R} \overline{v}u^{2}\;dx&=-\mbox{Re}\int\limits_{|x|\leq R}\Delta\chi_{R} \overline{v}u^{2}\;dx-\mbox{Re}\int\limits_{|x|\geq R}\Delta\chi_{R} \overline{v}u^{2}\;dx\\
&\leq-10\;\mbox{Re}\int\limits_{|x|\leq R} \overline{v}u^{2}\;dx+C_{1}\int\limits_{|x|\geq R}|v||u|^{2}\;dx\\
&\leq -10\;\mbox{Re}\int\limits_{\R^5} \overline{v}u^{2}\;dx+C_{1}'\int\limits_{|x|\geq R}|v||u|^{2}\;dx\\
&=5E(u,v)-5K(u,v)+C_{1}'\int\limits_{|x|\geq R}|v||u|^{2}\;dx.
\end{split}
\end{equation}
where $C_1'$ is also a positive constant.
Gathering together (\ref{estim1}), (\ref{estim2}) and (\ref{estim3}) and using the conservation  of the energy we get
\begin{equation}\label{1}
V''(t)\leq 5E(u_{0},v_{0})-K(u,v)+\frac{C_{2}'}{R^{2}}Q(u_{0},v_{0})+C_{1}'\int\limits_{|x|\geq R} |v||u|^{2}\;dx.
\end{equation}

Now, for the last integral in (\ref{1})  we use Young's inequality to obtain
$$C_{1}'\int\limits_{|x|\geq R} |v||u|^{2}\;dx\leq \frac{2C_{1}'}{3}\int\limits_{|x|\geq R}|u|^3\;dx+\frac{C_{1}'}{3}\int\limits_{|x|\geq R}|v|^3\;dx.$$
Then using  Lemma \ref{StraussLemaconsequence} and Young's inequality with $\epsilon$, we deduce that
 \begin{equation} \label{a1}
 \begin{split}
 C_{1}'\int\limits_{|x|\geq R}& |v||u|^{2}\;dx\leq\frac{C}{R^{2}}\|u\|^{5/2}_{L^{2}(|x|\geq R)}\|\nabla u\|^{1/2}_{L^{2}|x|\geq R)}
+\frac{C}{R^{2}}\|v\|^{5/2}_{L^{2}(|x|\geq R)}\|\nabla v\|^{1/2}_{L^{2}|x|\geq R)}\\
&\leq \frac{C_{\epsilon}}{R^{8/3}}\left(\|u\|^{10/3}_{L^{2}(|x|\geq R)}+\|v\|^{10/3}_{L^{2}(|x|\geq R)}\right)+\epsilon\left(\|\nabla u\|^{2}_{L^{2}(|x|\geq R)}+\frac{1}{2}\|\nabla v\|^{2}_{L^{2}(|x|\geq R)}\right)\\
&\leq\frac{\tilde{C}_{\epsilon}}{R^{8/3}}\left(\|u\|^{2}_{L^{2}(\R^{5})}+2\|v\|^{2}_{L^{2}(\R^{5})}\right)^{5/3}+\epsilon\left(\|\nabla u\|^{2}_{L^{2}(\R^{5})}+\frac{1}{2}\|\nabla v\|^{2}_{L^{2}(\R^{5})}\right)\\
&=\frac{\tilde{C}_{\epsilon}}{R^{8/3}}Q(u_{0},v_{0})^{5/3}+\epsilon K(u,v).
\end{split}
 \end{equation}
 Therefore, replacing \eqref{a1} into \eqref{1}, we conclude
\begin{equation}\label{estimaV2da}
V''(t)\leq 5E(u_{0},v_{0})-(1-\epsilon)K(u,v)+\frac{C_{2}}{R^{2}}Q(u_{0},v_{0})+\frac{\tilde{C}_{\epsilon}}{R^{8/3}}Q(u_{0},v_{0})^{5/3}.
\end{equation}
Multiplying (\ref{estimaV2da}) by $Q(u_{0},v_{0})$, we infer
\begin{equation*}
\begin{split}
Q(u_{0},v_{0})V''(t)&\leq 5E(u_{0},v_{0})Q(u_{0},v_{0})-(1-\epsilon)K(u,v)Q(u_{0},v_{0})\\
&\quad+\frac{C_{2}}{R^{2}}Q(u_{0},v_{0})^{2}+\frac{\tilde{C}_{\epsilon}}{R^{8/3}}Q(u_{0},v_{0})^{8/3}.
\end{split}
\end{equation*}
  Using (\ref{conddelta1}), (\ref{conddelta2}) and the fact that $E(\phi,\psi)=K(\phi,\psi)/5$, we finally obtain
\begin{equation*}
\begin{split}
Q(u_{0},v_{0})V''(t)&\leq (1-\delta_{1})Q(\phi,\psi)K(\phi,\psi)-(1-\epsilon)(1+\delta_{2})Q(\phi,\psi)K(\phi,\psi)\\
&\quad +\frac{C_{2}}{R^{2}}Q(u_{0},v_{0})^{2}+\frac{\tilde{C}_{\epsilon}}{R^{8/3}}Q(u_{0},v_{0})^{8/3}\\
&=[-\delta_{1}-\delta_2+\epsilon(1+\delta_{2})]Q(\phi,\psi)K(\phi,\psi)+\frac{C_{2}}{R^{2}}Q(u_{0},v_{0})^{2}+\frac{\tilde{C}_{\epsilon}}{R^{8/3}}Q(u_{0},v_{0})^{8/3}.
\end{split}
\end{equation*}
By choosing $\epsilon>0$ small enough and  $R>0$ large enough, we  conclude that $V''(t)<-B$, for some $B>0$. As before, this is enough to show that $I$ must be finite. 
\end{proof}

\begin{obs}
	If $E(u_0,v_0)<0$, then an application of Corollary \ref{corollarybestconstant} combined with the fact that $Q(\phi,\psi)=K(\phi,\psi)/5$  immediately gives
	$$
Q(u_0,v_0)K(u_0,v_0)>\left(\frac{5}{4}\right)^4	Q(\phi,\psi)K(\phi,\psi),
	$$
which in turn shows that \eqref{gradientcondblowup} holds. This is in agreement with the results in \cite{Hayashi} where the blow up was shown if the initial energy is negative (see also \cite{corcho}). However, since $(5/4)^4>1$, our result is stronger than the one \cite{Hayashi}.
\end{obs}

%%%%%%%%%%%%%%%%%%%%%%%%%%%%%%%%%%%%%%%%%%%%%%%%%%%%%%%%%%%%%%%%%%%%%%%%%%%%%%%%%%%%%%%%%%%

\section*{Acknowledgement}
N.N is supported by a partial scholarship from OAICE (University of Costa Rica). A.P. is partially supported by CNPq/Brazil grants 402849/2016-7 and 303098/2016-3.


\begin{thebibliography}{20}
	
\bibitem{beg} P. B\'egout, \emph{Necessary conditions and sufficient conditions for global existence in the nonlinear Schr\"odinger equation}, Adv. Math. Sci. Appl. 12 (2002), 817–827.

\bibitem{Cazenave} Th. Cazenave,\textit{Semilinear Schr\"odinger Equations}, Courant Lectures Notes in Mathematics, Vol. 10, American Mathematical Society, Providence, 2003. 

\bibitem{colin} M. Colin, L. Di Menza, and J.-C. Saut, \emph{Solitons in quadratic media},  Nonlinearity 29 (2016), 661--690.

\bibitem{corcho} A. Corcho, S. Correia, F. Oliveira, and J.D. Silva, \emph{On a nonlinear Schr\"odinger system arising in quadratic media}, arXiv 1703.10509v2.

\bibitem{Esfahani} A. Esfahani and A. Pastor, \emph{Sharp constant of an anisotropic Gagliardo-Nirenberg-type inequality and applications}, Bull. Braz. Math. Soc.  48 (2017), 171--185.
    
\bibitem{Evans} L. Evans, \textit{Partial Differential Equations}, Graduate Studies in Mathematics, Vol. 19, American Mathematical Society, Second Edition, Providence, 2010.
    
\bibitem{Hayashi3} N. Hayashi, C. Li, and P.I.  Naumkin, \emph{On a system of  nonlinear Schr\"{o}dinger equations  in 2D}, Differential Integral Equations  24 (2011), 417--434.

\bibitem{Hayashi4} N. Hayashi, C. Li, and T. Ozawa, \emph{Small data scattering for a system of nonlinear Schr\"odinger equations},  Differ. Equ. Appl.  3 (2011), 415--426.
				
\bibitem{Hayashi} N. Hayashi, T. Ozawa,  and K.  Tanaka,  \emph{On a system of nonlinear {S}chr\"odinger equations with quadratic interaction},  Ann. Inst. H. Poincar\'e Anal. Non Lin\'eaire 30 (2013), 661--690.
              
\bibitem{Holmer1} J. Holmer and S. Roudenko, \emph{On blow-up solutions to the 3D cubic nonlinear Schr\"odinger equation},  Appl. Math. Res. Express. AMRX  (2007), Art. ID abm004. 
      
\bibitem{Holmer2} J. Holmer and S. Roudenko, \emph{A sharp condition for scattering of the radial  3D cubic   nonlinear {S}chr\"odinger equation}, Commun. Math. Phys. 282 (2008),  435--467.        
 
\bibitem{hoshi} G. Hoshino and Ozawa,  \emph{Analytic smoothing effect for a system of Schr\"odinger equations with two wave interaction}, Adv. Differential Equations 20 (2015), 697--716.  
  
 \bibitem{Leoni} G. Leoni, \emph{A first course in Sobolev spaces}, Graduate Studies in Mathematics, Vol. 105, American Mathematical Society, Providence, 2009.
  
\bibitem{Li} C. Li,  \emph{On a system of nonlinear  Schr\"{o}dinger equations and scale invariant spaces in 2D}, Differential  Integral Equations 28 (2015), 201--220. 
    
\bibitem{LiHa} Li, C., Hayashi, N., \emph{Recent porgress on  nonlinear  Schr\"{o}dinger system with quadratic interaction } The Scientific World Journal 2014 (2014), Article ID 214821. 
   
    
\bibitem{Linares} F. Linares and G.   Ponce, \emph{Introduction to Nonlinear Dispersive Equations}, Universitex, Springer, New York, 2009.
    

\bibitem{Kavian} O. Kavian, \emph{A remark on the blowing-up solutions to the Cauchy problem for nonlinear Schr\"{o}dinger equations}, Trans. Am. Math. Soc. 299 (1987), 193--203.
	
	
\bibitem{Pastor} A.  Pastor, \emph{Weak concentration and wave operator for a 3D coupled nonlinear Schr\"{o}dinger system}, J. Math. Phys.  56 (2015), 021507-1 to 021507-18.
	

\bibitem{Ogawa} T. Ogawa and Y. Tsutsumi,  \emph{Blow-up of $H^{1}$ solutions for the nonlinear Schr\"{o}dinger equation}, J. Differential Equations 92 (1991), 317--330.

\bibitem{oza} T. Ozawa and H. Sunagawa, \emph{Small data blow-up for a system of nonlinear Schr\"odinger equations}, J. Math. Anal. Appl. 399 (2013),  147--155.

\bibitem{Strauss} W. Strauss, \emph{Existence of solitary waves in higher dimensions}, Commun. Math. Phys.  55 (1977), 149--162.

\bibitem{Weinstein} M.  Weinstein, \emph{ Nonlinear Schr\"{o}dinger equations and sharp interpolation estimates}, Commun. Math. Phys.  87 (1983), 567--576.
		
\end{thebibliography}
\end{document}